\documentclass[11pt]{amsart}

\usepackage[margin=1.0in]{geometry}
\usepackage{amsmath,amsthm,amssymb,amsfonts,mathtools}
\usepackage{enumitem}
\usepackage[backref=page]{hyperref}
\usepackage{mathrsfs}
\usepackage{tikz-cd}
\usepackage{microtype}
\usepackage{thmtools}
\usepackage{thm-restate}
\usepackage{array}
\usepackage{verbatim}
\usepackage{setspace}

\hypersetup{
  colorlinks=true,
  linkcolor=blue,
  citecolor=blue,
  urlcolor=blue
}

\title{On Chromatic Asymptotic Approximate Groups}
\author{Arindam Biswas}
\email{arin.math@gmail.com}
\date{}

\declaretheorem[style=plain,name=Theorem,numberwithin=section]{theorem}
\declaretheorem[style=plain,name=Proposition,sibling=theorem]{proposition}
\declaretheorem[style=plain,name=Lemma,sibling=theorem]{lemma}
\declaretheorem[style=plain,name=Corollary,sibling=theorem]{corollary}
\declaretheorem[style=definition,name=Definition,sibling=theorem]{definition}

\declaretheorem[style=remark,name=Remark,sibling=theorem]{remark}



\microtypesetup{nopatch=footnote}

\newcommand{\N}{\mathbb N}
\newcommand{\Nzero}{\mathbb N_0}
\newcommand{\Z}{\mathbb Z}
\newcommand{\R}{\mathbb R}
\newcommand{\cA}{\mathcal A}
\newcommand{\bh}{\mathbf h}
\newcommand{\bg}{\mathbf g}
\newcommand{\ba}{\mathbf a}

\newcommand{\precc}{\preceq}

\begin{document}

\begin{abstract}
We study a chromatic theory of asymptotic approximate groups for tuples of subsets of abelian groups, combining Nathanson's chromatic sumset formalism with asymptotic covering ideas from approximate group theory. This framework encodes simultaneous additive growth across several color classes. We show some general lifting and invariance principles, establish chromatic covering theorems for finite tuples and for tuples whose color classes are finite unions of unbounded linear sets, and obtain exact structure theorems for translated submonoids and finite-set-plus-submonoid sets. We also obtain sharper binomial bounds in the finite and unbounded-linear cases than the previous lattice-covering estimates. In the integer setting, we show that for each fixed threshold $t$, the threshold-$t$ chromatic layers form an asymptotic approximate family, using Nathanson's eventual interval-plus-edges description to obtain a uniform bound of size $r+2$ and prove an inhomogeneous extension for certain families.
\end{abstract}

\maketitle


\section{Introduction, motivation and background}

Let $G$ be an abelian group, written additively. For subsets $X,Y\subseteq G$, let $X+Y$ denote the Minkowski sumset
\[
X+Y:=\{x+y:x\in X,\ y\in Y\}.
\]
More generally, for subsets $X_1,\dots,X_m\subseteq G$, define
\[
X_1+\cdots+X_m:=\{x_1+\cdots+x_m:x_j\in X_j\text{ for all }j\}.
\]

If $A\subseteq G$ and $h\in \N$, the $h$-fold sumset of $A$ is
\[
hA:=\underbrace{A+\cdots+A}_{h\text{ times}},
\qquad
0A:=\{0\}.
\]

We begin with the additive form of Tao's notion of approximate group.

\begin{definition}\label{def:AG}
Let $K\ge 1$. A nonempty finite subset $A$ of an abelian group $G$ is called a $K$-approximate group if
\[
0\in A,
\qquad
-A=A,
\qquad\text{and}\qquad
A+A\subseteq X+A
\]
for some finite set $X\subseteq G$ with $|X|\le K$.
\end{definition}

Approximate groups are central objects in additive combinatorics, group theory, and number theory; see, for example, \cite{Hel08,BGT12}. Their modern formulation was motivated in part by their role, in a nonabelian setting, in Bourgain--Gamburd's work \cite{BG08} on super-strong approximation. In the abelian setting, related small-doubling phenomena were already present in classical work of Freiman \cite{Fre64}. Nathanson later introduced the more flexible notion of an asymptotic approximate group \cite{Nathanson2018}.

\begin{definition}[Asymptotic $(r,\ell)$-approximate group {\cite{Nathanson2018}}]\label{def:AAG}
Let $r,\ell\in \N$. A nonempty subset $A$ of an abelian group $G$ is called an asymptotic $(r,\ell)$-approximate group if there exists $h_0\in \N$ such that for every integer $h\ge h_0$, there exists a set $X_h\subseteq G$ with
\[
|X_h|\le \ell
\qquad\text{and}\qquad
r(hA)\subseteq X_h+hA.
\]
\end{definition}

Independently, Nathanson developed a chromatic sumset formalism in \cite{NathansonChromaticSumsets}. Fix a positive integer $q$, and let
\[
\cA=(A_1,\dots,A_q)
\]
be a $q$-tuple of nonempty subsets of $G$. Instead of a single scalar parameter $h$, one works with a vector parameter
\[
\bh=(h_1,\dots,h_q)\in \Nzero^q,
\]
and considers the associated chromatic sumset
\[
\bh\cdot \cA:=h_1A_1+\cdots+h_qA_q.
\]
We write
\[
\bh\precc \bh'
\]
if $h_i\le h_i'$ for every $i$, and for $r\in \N$ we write
\[
r\bh:=(rh_1,\dots,rh_q).
\]
In the integer setting, Nathanson also defined the chromatic representation function and proved that fixed-threshold chromatic layers eventually admit an interval-plus-edges description.

The present paper combines Nathanson's chromatic formalism with the asymptotic covering viewpoint. This leads naturally to the inclusion
\[
(r\bh)\cdot \cA\subseteq X_{\bh}+\bh\cdot \cA.
\]
Since the ambient group is abelian, this is equivalent to
\[
r(\bh\cdot \cA)\subseteq X_{\bh}+\bh\cdot \cA.
\]

\begin{definition}[CAAG]\label{def:CAAG}
Let $r,\ell\in \N$. Let $G$ be an abelian group and $\cA=(A_1,\dots,A_q)$ a $q$-tuple of nonempty subsets of $G$. We say that $\cA$ is a \emph{chromatic asymptotic $(r,\ell)$-approximate group}, abbreviated \emph{CAAG}, if there exists a threshold vector
\[
\bh_0\in \Nzero^q
\]
such that for every $\bh\succeq \bh_0$ there exists a set $X_{\bh}\subseteq G$ satisfying
\[
|X_{\bh}|\le \ell
\qquad\text{and}\qquad
(r\bh)\cdot \cA\subseteq X_{\bh}+\bh\cdot \cA.
\]
\end{definition}

\begin{definition}[CAG]\label{def:CAG}
If the same condition holds for every $\bh\in \Nzero^q$, we say that $\cA$ is a \emph{chromatic $(r,\ell)$-approximate group}.
\end{definition}

\begin{remark}
If $q=1$ and $\cA=(A)$, then $\bh\cdot \cA=hA$ and Definition~\ref{def:CAG} becomes
\[
rA\subseteq X+A,
\qquad |X|\le \ell,
\]
which is the one-set chromatic covering condition. After imposing the standard normalization conditions $0\in A$ and $-A=A$, this recovers the usual notion of approximate group from Definition~\ref{def:AG}. Likewise, Definition~\ref{def:CAAG} becomes: there exists $h_0\in \Nzero$ such that for every $h\ge h_0$, there is a set $X_h\subseteq G$ with
\[
r(hA)\subseteq X_h+hA,
\qquad |X_h|\le \ell,
\]
which is exactly the classical asymptotic $(r,\ell)$-approximate-group condition from Definition~\ref{def:AAG}.
\end{remark}

We shall also use linear and semilinear sets.

\begin{definition}
Let $a,b_1,\dots,b_d\in G$. The set
\[
P(a;b_1,\dots,b_d):=\{a+n_1b_1+\cdots+n_db_d:n_1,\dots,n_d\in \Nzero\}
\]
is called an \emph{unbounded generalized arithmetic progression} or \emph{unbounded linear set}.
\end{definition}

\begin{definition}
If bounds $m_1,\dots,m_d\in \Nzero$ are specified, the set
\[
P_{m_1,\dots,m_d}(a;b_1,\dots,b_d):=\{a+n_1b_1+\cdots+n_db_d:0\le n_j\le m_j\}
\]
is called a \emph{bounded generalized arithmetic progression} or \emph{bounded linear set}.
\end{definition}

\begin{definition}
A \emph{semilinear set} is a finite union of unbounded linear sets.
\end{definition}

Every bounded linear set is a finite union of singleton translates, hence a finite union of unbounded linear sets of the form $\{x\}=P(x;0)$. Therefore every finite union of bounded or unbounded linear sets is, after refinement, a finite union of unbounded linear sets.

For later quantitative statements, it is convenient to write
\[
\lambda_r(k):=\binom{(r+1)(k-1)}{k-1}
\qquad (r,k\in \N).
\]
The finite and unbounded-linear covering theorems will be expressed in terms of $\lambda_r(k)$.

We now record the representation-theoretic notation needed in the integer setting. For the rest of this paragraph, assume that $G=\Z$. Let $\cA=(A_1,\dots,A_q)$ be a $q$-tuple of finite nonempty subsets of $\Z$, and let $\bh=(h_1,\dots,h_q)\in \Nzero^q$. The \emph{chromatic representation function} $r_{\cA,\bh}(n)$ counts the number of tuples
\[
(a_{1,1},\dots,a_{1,h_1};\ a_{2,1},\dots,a_{2,h_2};\ \dots;\ a_{q,1},\dots,a_{q,h_q})
\]
such that
\begin{enumerate}[label=\textup{(\roman*)}]
\item $a_{i,j}\in A_i$ for every $i,j$,
\item within each color $i$, the sequence is weakly increasing,
\[
a_{i,1}\le a_{i,2}\le \cdots \le a_{i,h_i},
\]
\item the total sum equals $n$, that is,
\[
n=\sum_{i=1}^q\sum_{j=1}^{h_i} a_{i,j}.
\]
\end{enumerate}
For $t\in \N$, define the \emph{threshold-$t$ chromatic layer}
\[
(\bh\cdot \cA)^{(t)}:=\{n\in \Z:r_{\cA,\bh}(n)\ge t\}.
\]
For $t=1$, this is simply the chromatic sumset $\bh\cdot \cA$.

\begin{definition}
A finite integer tuple $\cA=(A_1,\dots,A_q)$ is called \emph{normalized} if
\[
\min(A_i)=0\quad\text{for all }i,
\]
and
\[
\gcd\Bigl(\bigcup_{i=1}^q A_i\Bigr)=1.
\]
\end{definition}

Finally, for a fixed nonempty set $B\subseteq G$, we consider inhomogeneous families of the form
\[
\bh\cdot \cA+B.
\]

\begin{definition}
Let $B\subseteq G$ be nonempty. We say that the inhomogeneous family $\bh\cdot \cA+B$ is an \emph{asymptotic $(r,\ell)$-approximate family} if there exists $\bh_0$ such that for every $\bh\succeq \bh_0$ there is a set $Y_{\bh}\subseteq G$ with
\[
|Y_{\bh}|\le \ell
\qquad\text{and}\qquad
r(\bh\cdot \cA+B)\subseteq Y_{\bh}+\bh\cdot \cA+B.
\]
\end{definition}

\subsection{Main results}

Nathanson proved in \cite{Nathanson2018} that every finite subset of an abelian group is an asymptotic approximate group. Biswas--Moens later gave a different proof, improved the published quantitative bounds, and extended the one-color theory to unions of unbounded linear sets and, after refinement, to semilinear sets \cite{BiswasMoens2022102330}. The present paper develops a chromatic analogue of this theory. In addition, our geometric covering argument yields sharper binomial constants than the cube-based bounds that appear in the one-color literature.

Our first result is a general lifting principle: chromatic asymptotic covering follows by combining one-color asymptotic coverings across the individual color classes.

\begin{theorem}[Lifting theorem for tuples]\label{thm:colorwise}
Let $G$ be an abelian group, let $\cA=(A_1,\dots,A_q)$, and fix $r\in \N$. Assume that for each $i$, the set $A_i$ is an asymptotic $(r,\ell_i)$-approximate group. Then $\cA$ is a chromatic asymptotic $(r,\prod_{i=1}^q \ell_i)$-approximate group.
\end{theorem}

For finite tuples, the chromatic covering constant factors over the colors. The one-color result is the binomial bound $\lambda_r(k)=\binom{(r+1)(k-1)}{k-1}$.

\begin{theorem}[Finite-tuple covering theorem]\label{thm:finitecolor}
Let $\cA=(A_1,\dots,A_q)$ be a tuple of finite subsets of an abelian group $G$, and write
\[
|A_i|=k_i
\qquad (1\le i\le q).
\]
Then for every $r\in \N$ and every $\bh\in \Nzero^q$, there exists a set $X_{\bh}\subseteq G$ such that
\[
(r\bh)\cdot \cA\subseteq X_{\bh}+\bh\cdot \cA
\]
and
\[
|X_{\bh}|
\le
\prod_{i=1}^q \lambda_r(k_i)
=
\prod_{i=1}^q \binom{(r+1)(k_i-1)}{k_i-1}.
\]
In particular, every finite tuple is a chromatic $(r,\ell)$-approximate group with
\[
\ell=
\prod_{i=1}^q \binom{(r+1)(k_i-1)}{k_i-1}.
\]
\end{theorem}

We next obtain exact structure when each color class is a translate of a submonoid.

\begin{theorem}\label{thm:submonoid}
Let $r\in \N$. Let $G$ be an abelian group. Suppose that for each color $i$,
\[
A_i=a_i+M_i,
\]
where $a_i\in G$ and $M_i\subseteq G$ is a submonoid, meaning that $0\in M_i$ and $M_i+M_i\subseteq M_i$. Then for every $\bh=(h_1,\dots,h_q)\in \Nzero^q$,
\[
(r\bh)\cdot \cA
=
(r-1)\sum_{i=1}^q h_i a_i+\bh\cdot \cA.
\]
Consequently $\cA$ is a chromatic $(r,1)$-approximate group.
\end{theorem}

More generally, one can adjoin finite cores to the colors while retaining a uniform chromatic covering theorem.

\begin{theorem}\label{thm:finiteplusmonoid}
Let $G$ be an abelian group. Suppose that for each color $i$,
\[
A_i=F_i+M_i,
\]
where $F_i\subseteq G$ is a finite nonempty set, $M_i\subseteq G$ is a submonoid, and
\[
k_i:=|F_i|.
\]
Then for every $r\in \N$ and every $\bh=(h_1,\dots,h_q)\in \Nzero^q$, there exists a set $X_{\bh}\subseteq G$ such that
\[
(r\bh)\cdot \cA\subseteq X_{\bh}+\bh\cdot \cA
\]
and
\[
|X_{\bh}|
\le
\prod_{i=1}^q \lambda_r(k_i)
=
\prod_{i=1}^q \binom{(r+1)(k_i-1)}{k_i-1}.
\]
In particular, $\cA$ is a chromatic $(r,\ell)$-approximate group with
\[
\ell=
\prod_{i=1}^q \binom{(r+1)(k_i-1)}{k_i-1}.
\]
\end{theorem}

The next theorem treats semilinear color classes. In the one-color case, Biswas--Moens proved asymptotic approximate-group structure for unions of unbounded linear sets \cite[Theorem 4.3]{BiswasMoens2022102330}. Here we obtain a chromatic version, together with the same sharper binomial constant furnished by the direct simplex-covering argument.

\begin{theorem}[Chromatic theorem for unions of unbounded linear sets]\label{thm:unboundedlinear}
Let $\cA=(A_1,\dots,A_q)$ be a tuple of subsets of an abelian group $G$. Assume that for each color $i$, the set $A_i$ is a union of $k_i$ unbounded linear sets. Then for every $r\in \N$, the tuple $\cA$ is a chromatic asymptotic $(r,\ell)$-approximate group with
\[
\ell=
\prod_{i=1}^q \lambda_r(k_i)
=
\prod_{i=1}^q \binom{(r+1)(k_i-1)}{k_i-1}.
\]
\end{theorem}

We also obtain an inhomogeneous extension.

\begin{theorem}\label{thm:inhomogeneous}
Let $\cA=(A_1,\dots,A_q)$ be a chromatic asymptotic $(r,\ell)$-approximate group in an abelian group $G$.
Let $B\subseteq G$ be a nonempty finite set of size $m$.
Then the family
\[
S_{\bh}:=\bh\cdot \cA+B
\]
is an asymptotic
\[
\left(r,\ell\,\lambda_r(m)\right)\text{-approximate family},
\]
that is, an asymptotic
\[
\left(r,\ell\binom{(r+1)(m-1)}{m-1}\right)\text{-approximate family}.
\]
\end{theorem}

Finally, we study threshold chromatic layers over the integers. Nathanson's theorem gives an eventual interval-plus-edges description of the threshold-$t$ chromatic layers for normalized finite tuples. We show that these layers themselves satisfy an asymptotic covering theorem with a uniform bound independent of the color parameter.

\begin{theorem}\label{thm:trich}
Let $\cA=(A_1,\dots,A_q)$ be a normalized tuple of finite sets of integers with $a_i^*=\max A_i\ge 1$ for every $i$. Fix $t,r\in \N$. Then there exists a threshold vector $\bh_{r,t}$ such that for every $\bh\succeq \bh_{r,t}$, there exists a set $X_{\bh}\subseteq \Z$ with
\[
|X_{\bh}|\le r+2
\]
and
\[
((r\bh)\cdot \cA)^{(t)}
\subseteq
X_{\bh}+(\bh\cdot \cA)^{(t)}.
\]
In other words, for each fixed $t$, the family of threshold-$t$ chromatic layers forms an asymptotic $(r,r+2)$-approximate family.
\end{theorem}

\section{Proofs}

\subsection{Basic identities and functorial properties}

\begin{lemma}\label{lem:basicidentity}
Let $G$ be abelian, let $\cA=(A_1,\dots,A_q)$, let $\bh\in \Nzero^q$, and let $r\in \N$. Then
\[
r(\bh\cdot \cA)=(r\bh)\cdot \cA.
\]
\end{lemma}

\begin{proof}
By the definition of repeated sumset,
\[
r(\bh\cdot \cA)=r(h_1A_1+\cdots+h_qA_q).
\]
Because the group is abelian, repeated sumsets distribute across finite sums:
\[
r(X_1+\cdots+X_q)=rX_1+\cdots+rX_q.
\]
Indeed, every element of the left-hand side has the form
\[
(x_{1,1}+\cdots+x_{q,1})+\cdots+(x_{1,r}+\cdots+x_{q,r})
=
(x_{1,1}+\cdots+x_{1,r})+\cdots+(x_{q,1}+\cdots+x_{q,r}),
\]
with $x_{i,j}\in X_i$, and conversely every element of $rX_1+\cdots+rX_q$ has this form.
Applying this with $X_i=h_iA_i$, we obtain
\[
r(h_1A_1+\cdots+h_qA_q)=r(h_1A_1)+\cdots+r(h_qA_q).
\]
But $r(h_iA_i)=rh_iA_i$, so
\[
r(\bh\cdot \cA)=rh_1A_1+\cdots+rh_qA_q=(r\bh)\cdot \cA.
\]
This proves the claim.
\end{proof}

\begin{lemma}
If $\cA$ is a CAAG with threshold $\bh_0$, then it is also a CAAG with any larger threshold $\bh_1\succeq \bh_0$.
\end{lemma}

\begin{proof}
This follows directly from the definition: if the covering condition holds for all $\bh\succeq \bh_0$, then it certainly holds for all $\bh\succeq \bh_1$.
\end{proof}

\begin{proposition}\label{prop:homomorphic}
Let $\phi:G\to H$ be a homomorphism of abelian groups. If $\cA=(A_1,\dots,A_q)$ is a chromatic asymptotic $(r,\ell)$-approximate group in $G$, then
\[
\phi(\cA):=(\phi(A_1),\dots,\phi(A_q))
\]
is a chromatic asymptotic $(r,\ell)$-approximate group in $H$.
\end{proposition}

\begin{proof}
Let $\bh_0$ be a threshold for $\cA$. Fix $\bh\succeq \bh_0$. By hypothesis there exists $X_{\bh}\subseteq G$ with
\[
|X_{\bh}|\le \ell
\qquad\text{and}\qquad
(r\bh)\cdot \cA\subseteq X_{\bh}+\bh\cdot \cA.
\]
Apply $\phi$. Since homomorphisms commute with finite sumsets,
\[
\phi((r\bh)\cdot \cA)=(r\bh)\cdot \phi(\cA),
\]
and likewise
\[
\phi(\bh\cdot \cA)=\bh\cdot \phi(\cA).
\]
Therefore
\[
(r\bh)\cdot \phi(\cA)
\subseteq \phi(X_{\bh})+\bh\cdot \phi(\cA).
\]
Since $|\phi(X_{\bh})|\le |X_{\bh}|\le \ell$, the result follows.
\end{proof}

\begin{proposition}\label{prop:affine}
Let $d\in \N$, let $b_1,\dots,b_q\in G$, and define
\[
A_i':=dA_i+b_i
\qquad (1\le i\le q).
\]
If $\cA=(A_1,\dots,A_q)$ is a chromatic asymptotic $(r,\ell)$-approximate group, then
\[
\cA':=(A_1',\dots,A_q')
\]
is also a chromatic asymptotic $(r,\ell)$-approximate group.
\end{proposition}

\begin{proof}
Fix $\bh=(h_1,\dots,h_q)$. Then
\[
\bh\cdot \cA'
=\sum_{i=1}^q h_i(dA_i+b_i)
=d\sum_{i=1}^q h_iA_i+\sum_{i=1}^q h_i b_i
=d(\bh\cdot \cA)+\beta_{\bh},
\]
where
\[
\beta_{\bh}:=\sum_{i=1}^q h_i b_i.
\]
Similarly,
\[
(r\bh)\cdot \cA'=d((r\bh)\cdot \cA)+r\beta_{\bh}.
\]
Suppose
\[
(r\bh)\cdot \cA\subseteq X_{\bh}+\bh\cdot \cA.
\]
Then
\[
d((r\bh)\cdot \cA)+r\beta_{\bh}
\subseteq dX_{\bh}+d(\bh\cdot \cA)+r\beta_{\bh}.
\]
Rewrite the right-hand side as
\[
\bigl(dX_{\bh}+(r-1)\beta_{\bh}\bigr)+\bigl(d(\bh\cdot \cA)+\beta_{\bh}\bigr)
=
\bigl(dX_{\bh}+(r-1)\beta_{\bh}\bigr)+\bh\cdot \cA'.
\]
Hence
\[
(r\bh)\cdot \cA'
\subseteq X_{\bh}'+\bh\cdot \cA'
\]
with
\[
X_{\bh}':=dX_{\bh}+(r-1)\beta_{\bh}.
\]
Since translation and the map $x\mapsto dx$ do not increase cardinality,
\[
|X_{\bh}'|\le |X_{\bh}|\le \ell.
\]
Thus $\cA'$ is a CAAG with the same parameters.
\end{proof}

\begin{proof}[Proof of Theorem~\ref{thm:colorwise}]
For each $i$, choose a threshold $h_{0,i}\in \Nzero$ such that whenever $h_i\ge h_{0,i}$, there exists $X_{i,h_i}\subseteq G$ satisfying
\[
|X_{i,h_i}|\le \ell_i
\qquad\text{and}\qquad
r(h_iA_i)\subseteq X_{i,h_i}+h_iA_i.
\]
Set
\[
\bh_0:=(h_{0,1},\dots,h_{0,q}).
\]
Let $\bh=(h_1,\dots,h_q)\succeq \bh_0$. Then for each color $i$,
\[
r(h_iA_i)\subseteq X_{i,h_i}+h_iA_i.
\]
Summing these inclusions over $i$, we obtain
\[
\sum_{i=1}^q r(h_iA_i)
\subseteq
\sum_{i=1}^q (X_{i,h_i}+h_iA_i).
\]
Since finite sums distribute,
\[
\sum_{i=1}^q (X_{i,h_i}+h_iA_i)
=
\Bigl(\sum_{i=1}^q X_{i,h_i}\Bigr)+\Bigl(\sum_{i=1}^q h_iA_i\Bigr).
\]
By Lemma~\ref{lem:basicidentity},
\[
(r\bh)\cdot \cA
=\sum_{i=1}^q r(h_iA_i),
\]

\[
\bh\cdot \cA=\sum_{i=1}^q h_iA_i.
\]
Therefore
\[
(r\bh)\cdot \cA
\subseteq X_{\bh}+\bh\cdot \cA,
\]
where
\[
X_{\bh}:=\sum_{i=1}^q X_{i,h_i}.
\]
Finally,
\[
|X_{\bh}|
\le \prod_{i=1}^q |X_{i,h_i}|
\le \prod_{i=1}^q \ell_i.
\]
Hence $\cA$ is a chromatic asymptotic $(r,\prod_i \ell_i)$-approximate group.
\end{proof}

\subsection{Finite color classes}

In this subsection we show that finite tuples are chromatic approximate groups with explicit uniform bounds.

\subsubsection{A direct lattice-covering lemma for simplices}

For $d\in \N$ and $\rho\ge 0$, define
\[
\Delta_d(\rho):=\{(x_1,\dots,x_d)\in \R_{\ge 0}^d:x_1+\cdots+x_d\le \rho\}.
\]
For $t>0$, define
\[
\Delta_d(t)^\circ:=\{(x_1,\dots,x_d)\in \R_{\ge 0}^d:x_1+\cdots+x_d<t\}.
\]

\begin{lemma}\label{lem:latticecover}
Let $d,R,t\in \N$. Put
\[
M:=\left\lfloor \frac{dR}{t}\right\rfloor.
\]
Then
\[
\Delta_d(R)\cap \Z^d
\]
can be covered by at most
\[
\binom{M+d}{d}
\]
translates, by elements of $\Z^d$, of
\[
\Delta_d(t)^\circ\cap \Z^d.
\]
\end{lemma}

\begin{proof}
Fix
\[
x=(x_1,\dots,x_d)\in \Delta_d(R)\cap \Z^d.
\]
For each $i\in [1,d]$, define
\[
m_i:=\left\lfloor \frac{d x_i}{t}\right\rfloor\in \Nzero.
\]
Then
\[
m_i\le \frac{d x_i}{t},
\]
and therefore
\[
m_1+\cdots+m_d
\le
\frac{d}{t}(x_1+\cdots+x_d)
\le
\frac{dR}{t}.
\]
Since the left-hand side is an integer, it follows that
\[
m_1+\cdots+m_d\le M.
\]

Now define
\[
u_i:=\left\lceil \frac{t m_i}{d}\right\rceil\in \Nzero
\qquad (1\le i\le d),
\]
and set
\[
u:=(u_1,\dots,u_d)\in \Z^d.
\]
Because
\[
\frac{t m_i}{d}\le x_i
\]
and $x_i$ is an integer, we have
\[
u_i=\left\lceil \frac{t m_i}{d}\right\rceil \le x_i.
\]
Hence
\[
x_i-u_i\in \Nzero
\qquad\text{for all }i.
\]
Moreover, since $m_i=\lfloor d x_i/t\rfloor$, we have
\[
m_i>\frac{d x_i}{t}-1,
\]
and multiplying by $t/d$ gives
\[
\frac{t m_i}{d}>x_i-\frac{t}{d}.
\]
Therefore
\[
u_i\ge \frac{t m_i}{d}>x_i-\frac{t}{d},
\]
so
\[
0\le x_i-u_i<\frac{t}{d}.
\]
Summing over $i$, we obtain
\[
0\le \sum_{i=1}^d (x_i-u_i)< d\cdot \frac{t}{d}=t.
\]
Thus
\[
x-u\in \Delta_d(t)^\circ\cap \Z^d,
\]
and so
\[
x\in u+\bigl(\Delta_d(t)^\circ\cap \Z^d\bigr).
\]

We have shown that every $x\in \Delta_d(R)\cap \Z^d$ lies in a translate indexed by a vector
\[
m=(m_1,\dots,m_d)\in \Nzero^d
\qquad\text{with}\qquad
m_1+\cdots+m_d\le M.
\]
The number of such vectors is
\[
\sum_{s=0}^M \binom{s+d-1}{d-1}
=
\binom{M+d}{d}.
\]
Hence the stated covering number bound follows.
\end{proof}

\subsubsection{Covering theorem in a free abelian group}

Let $\Z^k$ have standard basis $e_1,\dots,e_k$, and let
\[
B:=\{e_1,\dots,e_k\}.
\]
Then
\[
hB=\{x\in \Z_{\ge 0}^k:x_1+\cdots+x_k=h\}.
\]

\begin{proposition}\label{prop:freefinite}
Let $k,r\in \N$. For every $h\in \Nzero$ there exists a set $Y_h\subseteq \Z^k$ with
\[
|Y_h|\le \binom{(r+1)(k-1)}{k-1}
\]
such that
\[
rhB\subseteq Y_h+hB.
\]
\end{proposition}

\begin{proof}
If $h=0$, then
\[
rhB=hB=\{0\},
\]
so the claim is immediate.

If $k=1$, then
\[
B=\{e_1\},\qquad hB=\{he_1\},\qquad rhB=\{rhe_1\},
\]
and hence
\[
rhB\subseteq \{(r-1)he_1\}+hB.
\]
Since
\[
\binom{(r+1)(1-1)}{1-1}=\binom00=1,
\]
the required bound holds.

Assume now that $k\ge 2$ and $h\ge 1$, and write
\[
d:=k-1.
\]
Apply Lemma~\ref{lem:latticecover} with
\[
R=rh,
\qquad
t=h.
\]
Then
\[
M=\left\lfloor \frac{d(rh)}{h}\right\rfloor=dr,
\]
so there exist vectors
\[
v_1,\dots,v_N\in \Z^d,
\qquad
N\le \binom{dr+d}{d}=\binom{(r+1)(k-1)}{k-1},
\]
such that
\[
\Delta_d(rh)\cap \Z^d
\subseteq
\bigcup_{j=1}^N \left(v_j+\bigl(\Delta_d(h)^\circ\cap \Z^d\bigr)\right).
\]

Consider the affine map
\[
f:\Z^d\to \Z^k,
\qquad
f(y_1,\dots,y_d):=(y_1,\dots,y_d,\,rh-(y_1+\cdots+y_d)).
\]
Its image on
\[
\Delta_d(rh)\cap \Z^d
\]
is exactly $rhB$.

Fix $j\in [1,N]$. For
\[
z=(z_1,\dots,z_d)\in \Delta_d(h)^\circ\cap \Z^d,
\]
define
\[
g(z):=(z_1,\dots,z_d,\,h-(z_1+\cdots+z_d)).
\]
Since
\[
z_1+\cdots+z_d<h
\]
and the left-hand side is an integer, we have
\[
h-(z_1+\cdots+z_d)\in \N,
\]
so $g(z)\in hB$.

Now compute
\[
f(v_j+z)
=
(v_{j,1}+z_1,\dots,v_{j,d}+z_d,\,
rh-\sum_{m=1}^d v_{j,m}-\sum_{m=1}^d z_m),
\]
while
\[
\bigl(f(v_j)-(0,\dots,0,h)\bigr)+g(z)
=
(v_{j,1}+z_1,\dots,v_{j,d}+z_d,\,
rh-\sum_{m=1}^d v_{j,m}-h+h-\sum_{m=1}^d z_m).
\]
Hence
\[
f(v_j+z)=\bigl(f(v_j)-(0,\dots,0,h)\bigr)+g(z).
\]
Therefore
\[
f\!\left(v_j+\bigl(\Delta_d(h)^\circ\cap \Z^d\bigr)\right)
\subseteq
\bigl(f(v_j)-(0,\dots,0,h)\bigr)+hB.
\]

Taking the union over $j=1,\dots,N$, we obtain
\[
rhB\subseteq Y_h+hB,
\]
where
\[
Y_h:=\{\,f(v_j)-(0,\dots,0,h):1\le j\le N\,\}.
\]
Thus
\[
|Y_h|\le N\le \binom{(r+1)(k-1)}{k-1}.
\]
This proves the proposition.
\end{proof}

\subsubsection{Finite subsets of an arbitrary abelian group}

\begin{proposition}\label{prop:finiteabelian}
Let $A\subseteq G$ be finite, with $|A|=k$. Then for every $r\in \N$ and every $h\in \Nzero$, there exists a set $X_h\subseteq G$ with
\[
|X_h|\le \binom{(r+1)(k-1)}{k-1}
\]
such that
\[
r(hA)\subseteq X_h+hA.
\]
\end{proposition}

\begin{proof}
Enumerate the elements of $A$ as
\[
A=\{a_1,\dots,a_k\}.
\]
Let $\pi:\Z^k\to G$ be the homomorphism defined by
\[
\pi(e_j)=a_j
\qquad (1\le j\le k).
\]
Then
\[
\pi(B)=A,
\qquad
\pi(hB)=hA,
\qquad
\pi(rhB)=r(hA).
\]
By Proposition~\ref{prop:freefinite}, there exists $Y_h\subseteq \Z^k$ such that
\[
|Y_h|\le \binom{(r+1)(k-1)}{k-1}
\qquad\text{and}\qquad
rhB\subseteq Y_h+hB.
\]
Applying $\pi$, we obtain
\[
r(hA)=\pi(rhB)\subseteq \pi(Y_h)+\pi(hB)=X_h+hA,
\]
where
\[
X_h:=\pi(Y_h).
\]
Since $|X_h|\le |Y_h|$, the result follows.
\end{proof}

\begin{proof}[Proof of Theorem~\ref{thm:finitecolor}]
Fix $\bh=(h_1,\dots,h_q)$. For each color $i$, Proposition~\ref{prop:finiteabelian} gives a set $X_{i,h_i}\subseteq G$ with
\[
|X_{i,h_i}|
\le
\binom{(r+1)(k_i-1)}{k_i-1}
\qquad\text{and}\qquad
r(h_iA_i)\subseteq X_{i,h_i}+h_iA_i.
\]
Summing these inclusions over all $i$,
\[
\sum_{i=1}^q r(h_iA_i)
\subseteq
\sum_{i=1}^q (X_{i,h_i}+h_iA_i)
=
\Bigl(\sum_{i=1}^q X_{i,h_i}\Bigr)+\Bigl(\sum_{i=1}^q h_iA_i\Bigr).
\]
By Lemma~\ref{lem:basicidentity}, the left-hand side equals $(r\bh)\cdot \cA$, while the second sum on the right is $\bh\cdot \cA$. Hence
\[
(r\bh)\cdot \cA\subseteq X_{\bh}+\bh\cdot \cA,
\]
where
\[
X_{\bh}:=\sum_{i=1}^q X_{i,h_i}.
\]
The cardinality satisfies
\[
|X_{\bh}|
\le
\prod_{i=1}^q |X_{i,h_i}|
\le
\prod_{i=1}^q \binom{(r+1)(k_i-1)}{k_i-1}.
\]
This proves the theorem.
\end{proof}

\begin{remark}
This is stronger than asymptoticity: no threshold is needed. The same uniform bound works for every chromatic parameter $\bh$.
\end{remark}

\subsection{Linear, semilinear, and inhomogeneous chromatic families}

\begin{proof}[Proof of Theorem~\ref{thm:submonoid}]
Fix $\bh$.

If $\bh=\mathbf{0}$, then both sides are equal to $\{0\}$, so the claim follows immediately.
Assume from now on that $\bh\neq \mathbf{0}$. For each color $i$, we distinguish the cases $h_i=0$ and $h_i\ge 1$.

If $h_i=0$, then by convention
\[
h_iA_i=0A_i=\{0\}.
\]
If $h_i\ge 1$, then
\[
h_iA_i=h_i(a_i+M_i)=h_i a_i+h_iM_i.
\]
Because $M_i$ is a submonoid and contains $0$, we have $h_iM_i=M_i$ for every $h_i\ge 1$: inclusion $M_i\subseteq h_iM_i$ follows by inserting zero terms, while $h_iM_i\subseteq M_i$ follows from closure under addition. Thus
\[
h_iA_i=
\begin{cases}
\{0\}, & h_i=0,\\
h_i a_i+M_i, & h_i\ge 1.
\end{cases}
\]
Define
\[
I(\bh):=\{i:h_i\ge 1\}.
\]
Then
\[
\bh\cdot \cA
=\sum_{i\in I(\bh)}(h_i a_i+M_i)
=\alpha_{\bh}+M_{\bh},
\]
where
\[
\alpha_{\bh}:=\sum_{i=1}^q h_i a_i,
\qquad
M_{\bh}:=\sum_{i\in I(\bh)} M_i.
\]
The same reasoning gives
\[
(r\bh)\cdot \cA=r\alpha_{\bh}+M_{\bh}.
\]
Hence
\[
(r\bh)\cdot \cA
=r\alpha_{\bh}+M_{\bh}
=(r-1)\alpha_{\bh}+(\alpha_{\bh}+M_{\bh})
=(r-1)\sum_{i=1}^q h_i a_i+\bh\cdot \cA.
\]
Therefore the covering holds with the singleton
\[
X_{\bh}:=\left\{(r-1)\sum_{i=1}^q h_i a_i\right\}.
\]
So $|X_{\bh}|=1$.
\end{proof}

\begin{definition}
Let $K\in \N$. A subset $M$ of an abelian group $G$ is called a \emph{$K$-approximate submonoid} if
\[
0\in M
\]
and there exists a finite set $F\subseteq G$ with
\[
|F|\le K
\qquad\text{and}\qquad
M+M\subseteq F+M.
\]
\end{definition}

The exact one-translate identity of Theorem~\ref{thm:submonoid} need not persist for approximate submonoids. However, one still obtains a uniform finite covering theorem once the approximate-submonoid parameters are fixed.

\begin{proposition}[Approximate submonoid variant]\label{prop:approxsubmonoid}
Let $r\in \N$, let $G$ be an abelian group, and let
\[
\cA=(A_1,\dots,A_q).
\]
Suppose that for each color $i$,
\[
A_i=a_i+M_i,
\]
where $a_i\in G$ and $M_i\subseteq G$ is a $K_i$-approximate submonoid for some $K_i\in \N$.
Then for every $\bh=(h_1,\dots,h_q)\in \Nzero^q$, there exists a finite set $X_{\bh}\subseteq G$ such that
\[
(r\bh)\cdot \cA\subseteq X_{\bh}+\bh\cdot \cA
\]
and
\[
|X_{\bh}|
\le
\prod_{i:\,h_i>0}
\binom{(r-1)h_i+K_i-1}{K_i-1}.
\]
Equivalently, if $K:=\max_i K_i$, then
\[
|X_{\bh}|
\le
\prod_{i:\,h_i>0}
\binom{(r-1)h_i+K-1}{K-1}.
\]
\end{proposition}

\begin{proof}
For each color $i$, choose a finite set $F_i\subseteq G$ such that
\[
|F_i|\le K_i
\qquad\text{and}\qquad
M_i+M_i\subseteq F_i+M_i.
\]

Fix $\bh=(h_1,\dots,h_q)\in \Nzero^q$. We first construct, for each color $i$, a finite set $X_{i,h_i}$ satisfying
\[
r(h_iA_i)\subseteq X_{i,h_i}+h_iA_i.
\]

If $h_i=0$, then
\[
h_iA_i=0A_i=\{0\},
\qquad
r(h_iA_i)=\{0\},
\]
so we may take
\[
X_{i,0}:=\{0\}.
\]

Assume now that $h_i\ge 1$. We claim that for every integer $s\in \Nzero$,
\[
sM_i+h_iM_i\subseteq sF_i+h_iM_i.
\]
We prove this by induction on $s$.

For $s=0$, the statement is immediate. Suppose it holds for some $s\ge 0$. Since $h_i\ge 1$, we have
\[
M_i+h_iM_i=(M_i+M_i)+(h_i-1)M_i\subseteq F_i+M_i+(h_i-1)M_i=F_i+h_iM_i.
\]
Therefore
\[
(s+1)M_i+h_iM_i
=
sM_i+(M_i+h_iM_i)
\subseteq
sM_i+F_i+h_iM_i
\subseteq
sF_i+F_i+h_iM_i
=
(s+1)F_i+h_iM_i.
\]
This completes the induction.

Applying the claim with
\[
s=(r-1)h_i,
\]
we obtain
\[
rh_iM_i=(r-1)h_iM_i+h_iM_i\subseteq (r-1)h_iF_i+h_iM_i.
\]
Hence
\[
r(h_iA_i)
=
r(h_i a_i+h_iM_i)
=
rh_i a_i+rh_iM_i
\subseteq
\bigl((r-1)h_i a_i+(r-1)h_iF_i\bigr)+\bigl(h_i a_i+h_iM_i\bigr).
\]
Since
\[
h_iA_i=h_i a_i+h_iM_i,
\]
this yields
\[
r(h_iA_i)\subseteq X_{i,h_i}+h_iA_i,
\]
where
\[
X_{i,h_i}:=
(r-1)h_i a_i+(r-1)h_iF_i.
\]

We now bound the cardinality of $X_{i,h_i}$. Translation does not change cardinality, so
\[
|X_{i,h_i}|=|(r-1)h_iF_i|.
\]
Write
\[
F_i=\{f_{i,1},\dots,f_{i,m_i}\},
\qquad m_i:=|F_i|\le K_i.
\]
Every element of $(r-1)h_iF_i$ can be written in the form
\[
n_1f_{i,1}+\cdots+n_{m_i}f_{i,m_i},
\qquad
n_1,\dots,n_{m_i}\in \Nzero,
\qquad
n_1+\cdots+n_{m_i}=(r-1)h_i.
\]
The number of such weak compositions is
\[
\binom{(r-1)h_i+m_i-1}{m_i-1},
\]
and therefore
\[
|X_{i,h_i}|
\le
\binom{(r-1)h_i+m_i-1}{m_i-1}
\le
\binom{(r-1)h_i+K_i-1}{K_i-1}.
\]

Thus, in all cases,
\[
r(h_iA_i)\subseteq X_{i,h_i}+h_iA_i
\]
with
\[
|X_{i,h_i}|
\le
\binom{(r-1)h_i+K_i-1}{K_i-1}
\qquad (h_i\ge 1),
\]
and $|X_{i,0}|=1$.

Summing the colorwise inclusions over $i$, we obtain
\[
\sum_{i=1}^q r(h_iA_i)
\subseteq
\sum_{i=1}^q (X_{i,h_i}+h_iA_i)
=
\Bigl(\sum_{i=1}^q X_{i,h_i}\Bigr)+\Bigl(\sum_{i=1}^q h_iA_i\Bigr).
\]
By Lemma~\ref{lem:basicidentity}, the left-hand side is $(r\bh)\cdot \cA$, while the second sum on the right is $\bh\cdot \cA$. Hence
\[
(r\bh)\cdot \cA\subseteq X_{\bh}+\bh\cdot \cA,
\qquad
X_{\bh}:=\sum_{i=1}^q X_{i,h_i}.
\]
Finally,
\[
|X_{\bh}|
\le
\prod_{i=1}^q |X_{i,h_i}|
\le
\prod_{i:\,h_i>0}
\binom{(r-1)h_i+K_i-1}{K_i-1}.
\]
This proves the proposition.
\end{proof}

\begin{remark}
Unlike Theorem~\ref{thm:submonoid}, the covering bound in Proposition~\ref{prop:approxsubmonoid} generally depends on $\bh$. Thus the conclusion is not, in general, a chromatic $(r,\ell)$-approximate-group statement with $\ell$ independent of $\bh$. The proposition should instead be viewed as a polynomial covering theorem whose constants are uniform in the choice of the sets $A_i$ once the parameters $K_i$ are fixed.
\end{remark}

\begin{proof}[Proof of Theorem~\ref{thm:finiteplusmonoid}]
Fix $\bh=(h_1,\dots,h_q)$. For each color $i$, we define a finite set $X_{i,h_i}$.

If $h_i=0$, then
\[
h_iA_i=0A_i=\{0\},
\qquad
r(h_iA_i)=\{0\},
\]
so the covering holds with
\[
X_{i,0}:=\{0\}.
\]

Assume now that $h_i\ge 1$. Since $M_i$ is a submonoid and contains $0$, we have
\[
h_iM_i=M_i.
\]
Therefore
\[
h_iA_i=h_i(F_i+M_i)=h_iF_i+h_iM_i=h_iF_i+M_i.
\]
By Proposition~\ref{prop:finiteabelian}, there exists a finite set $Y_{i,h_i}\subseteq G$ such that
\[
|Y_{i,h_i}|
\le
\binom{(r+1)(k_i-1)}{k_i-1}
\qquad\text{and}\qquad
r(h_iF_i)\subseteq Y_{i,h_i}+h_iF_i.
\]
Hence
\[
r(h_iA_i)
=
r(h_iF_i+M_i)
=
r(h_iF_i)+rM_i
=
r(h_iF_i)+M_i
\subseteq
Y_{i,h_i}+h_iF_i+M_i
=
Y_{i,h_i}+h_iA_i.
\]
Set
\[
X_{i,h_i}:=Y_{i,h_i}.
\]
Then in all cases
\[
|X_{i,h_i}|
\le
\binom{(r+1)(k_i-1)}{k_i-1}
\qquad\text{and}\qquad
r(h_iA_i)\subseteq X_{i,h_i}+h_iA_i.
\]

Summing over all colors gives
\[
\sum_{i=1}^q r(h_iA_i)
\subseteq
\sum_{i=1}^q (X_{i,h_i}+h_iA_i)
=
\Bigl(\sum_{i=1}^q X_{i,h_i}\Bigr)+\Bigl(\sum_{i=1}^q h_iA_i\Bigr).
\]
By Lemma~\ref{lem:basicidentity}, the left-hand side is $(r\bh)\cdot \cA$ and the second sum on the right is $\bh\cdot \cA$. Therefore
\[
(r\bh)\cdot \cA\subseteq X_{\bh}+\bh\cdot \cA,
\qquad
X_{\bh}:=\sum_{i=1}^q X_{i,h_i}.
\]
Finally,
\[
|X_{\bh}|
\le
\prod_{i=1}^q |X_{i,h_i}|
\le
\prod_{i=1}^q \binom{(r+1)(k_i-1)}{k_i-1}.
\]
This proves the theorem.
\end{proof}

\begin{corollary}\label{cor:commonmonoid}
Let $G$ be an abelian group. Suppose that for each color $i$ there exists a finite set
\[
F_i=\{f_{i,1},\dots,f_{i,k_i}\}\subseteq G
\]
and a submonoid $M_i\subseteq G$ such that
\[
A_i=\bigcup_{j=1}^{k_i}(f_{i,j}+M_i)=F_i+M_i.
\]
Then for every $r\in \N$, the tuple $\cA=(A_1,\dots,A_q)$ is a chromatic
\[
\left(r,\prod_{i=1}^q \binom{(r+1)(k_i-1)}{k_i-1}\right)\text{-approximate group}.
\]
\end{corollary}

\begin{proof}
This is immediate from Theorem~\ref{thm:finiteplusmonoid}.
\end{proof}

We next prove a one-color covering theorem for unions of unbounded linear sets.

\begin{proposition}[Covering theorem for positive composition shells]\label{prop:positiveshelltuple}
Let $G$ be an abelian group, and let $f_1,\dots,f_k\in G$.
For $h\in \Nzero$, define
\[
\Sigma_h(f_1,\dots,f_k)
:=
\left\{
n_1f_1+\cdots+n_kf_k:
n_1,\dots,n_k\in \Nzero,\ 
n_1+\cdots+n_k=h
\right\}.
\]
If $h\ge k$, define also
\[
\Sigma_h^+(f_1,\dots,f_k)
:=
\left\{
n_1f_1+\cdots+n_kf_k:
n_1,\dots,n_k\in \N,\ 
n_1+\cdots+n_k=h
\right\}.
\]

If $k=1$, then for every $h\in \N$,
\[
\Sigma_{rh}(f_1)=\{(r-1)hf_1\}+\Sigma_h^+(f_1).
\]

Assume now that $k\ge 2$ and
\[
h\ge r(k-1)^2+k.
\]
Then there exists a set $X_h\subseteq G$ such that
\[
|X_h|\le \binom{(r+1)(k-1)}{k-1}
\]
and
\[
\Sigma_{rh}(f_1,\dots,f_k)
\subseteq
X_h+\Sigma_h^+(f_1,\dots,f_k).
\]
\end{proposition}

\begin{proof}
If $k=1$, the displayed identity is immediate.

Assume that $k\ge 2$, and write
\[
d:=k-1.
\]
Let
\[
B:=\{e_1,\dots,e_k\}\subseteq \Z^k,
\]
and define
\[
hB^+:=\{x\in \Z_{\ge 1}^k:x_1+\cdots+x_k=h\}.
\]
Apply Lemma~\ref{lem:latticecover} with
\[
R=rh,
\qquad
t=h-d.
\]
Since
\[
h\ge rd^2+d+1,
\]
we have
\[
t=h-d\ge rd^2+1\ge 1.
\]
Moreover,
\[
\frac{dR}{t}
=
\frac{drh}{h-d}
<
dr+1,
\]
because
\[
drh<(dr+1)(h-d)
\iff
0<h-d(dr+1).
\]
Hence
\[
\left\lfloor \frac{dR}{t}\right\rfloor \le dr.
\]
By Lemma~\ref{lem:latticecover}, there exist vectors
\[
v_1,\dots,v_N\in \Z^d,
\qquad
N\le \binom{dr+d}{d}=\binom{(r+1)(k-1)}{k-1},
\]
such that
\[
\Delta_d(rh)\cap \Z^d
\subseteq
\bigcup_{j=1}^N
\left(
v_j+\bigl(\Delta_d(h-d)^\circ\cap \Z^d\bigr)
\right).
\]

Consider the affine map
\[
f:\Z^d\to \Z^k,
\qquad
f(y_1,\dots,y_d):=(y_1,\dots,y_d,\,rh-(y_1+\cdots+y_d)).
\]
Its image on
\[
\Delta_d(rh)\cap \Z^d
\]
is exactly $rhB$.

For
\[
z=(z_1,\dots,z_d)\in \Delta_d(h-d)^\circ\cap \Z^d,
\]
define
\[
g(z):=(z_1+1,\dots,z_d+1,\,h-d-(z_1+\cdots+z_d)).
\]
Since
\[
z_1+\cdots+z_d<h-d
\]
and the left-hand side is an integer, the last coordinate of $g(z)$ is a positive integer. Therefore
\[
g(z)\in hB^+.
\]

Now compute
\[
f(v_j+z)
=
(v_{j,1}+z_1,\dots,v_{j,d}+z_d,\,
rh-\sum_{m=1}^d v_{j,m}-\sum_{m=1}^d z_m),
\]
while
\[
\bigl(f(v_j)-(1,\dots,1,h-d)\bigr)+g(z)
=
(v_{j,1}+z_1,\dots,v_{j,d}+z_d,\,
rh-\sum_{m=1}^d v_{j,m}-(h-d)+h-d-\sum_{m=1}^d z_m).
\]
Hence
\[
f(v_j+z)=\bigl(f(v_j)-(1,\dots,1,h-d)\bigr)+g(z).
\]
Therefore
\[
rhB\subseteq Y_h+hB^+,
\]
where
\[
Y_h:=\{\,f(v_j)-(1,\dots,1,h-d):1\le j\le N\,\}.
\]
Thus
\[
|Y_h|\le N\le \binom{(r+1)(k-1)}{k-1}.
\]

Finally, let
\[
\pi:\Z^k\to G
\]
be the homomorphism defined by
\[
\pi(e_i)=f_i
\qquad (1\le i\le k).
\]
Then
\[
\pi(rhB)=\Sigma_{rh}(f_1,\dots,f_k)
\]
and
\[
\pi(hB^+)=\Sigma_h^+(f_1,\dots,f_k).
\]
Applying $\pi$ to the inclusion
\[
rhB\subseteq Y_h+hB^+,
\]
we obtain
\[
\Sigma_{rh}(f_1,\dots,f_k)
\subseteq
X_h+\Sigma_h^+(f_1,\dots,f_k),
\]
where
\[
X_h:=\pi(Y_h).
\]
Since $|X_h|\le |Y_h|$, the required bound follows.
\end{proof}

\begin{theorem}[One-color theorem for unions of unbounded linear sets]\label{thm:onecolorlinear}
Let $A\subseteq G$ be a union of $k$ unbounded linear sets in an abelian group $G$. Then for every $r\in \N$, the set $A$ is an asymptotic
\[
\left(r,\binom{(r+1)(k-1)}{k-1}\right)\text{-approximate group}.
\]
\end{theorem}

\begin{proof}
Write
\[
A=\bigcup_{i=1}^k (a_i+M_i),
\]
where each $a_i\in G$ and each $M_i\subseteq G$ is a submonoid.
Set
\[
M:=M_1+\cdots+M_k.
\]
Fix $h\ge r(k-1)^2+k$.
By Proposition~\ref{prop:positiveshelltuple}, there exists a set $X_h\subseteq G$ such that
\[
|X_h|\le \binom{(r+1)(k-1)}{k-1}
\]
and
\[
\Sigma_{rh}(a_1,\dots,a_k)
\subseteq
X_h+\Sigma_h^+(a_1,\dots,a_k).
\]

We claim that
\[
rhA\subseteq \Sigma_{rh}(a_1,\dots,a_k)+M
\]
and
\[
\Sigma_h^+(a_1,\dots,a_k)+M\subseteq hA.
\]

For the first claim, let $x\in rhA$. Then
\[
x\in n_1(a_1+M_1)+\cdots+n_k(a_k+M_k)
\]
for some $n_1,\dots,n_k\in \Nzero$ with
\[
n_1+\cdots+n_k=rh.
\]
Hence
\[
x\in
(n_1a_1+\cdots+n_ka_k)
+
(n_1M_1+\cdots+n_kM_k).
\]
If $n_i=0$, then $n_iM_i=\{0\}\subseteq M_i$.
If $n_i\ge 1$, then $n_iM_i=M_i$ because $M_i$ is a submonoid containing $0$.
Therefore
\[
n_1M_1+\cdots+n_kM_k\subseteq M,
\]
and so
\[
x\in \Sigma_{rh}(a_1,\dots,a_k)+M.
\]
This proves the first claim.

For the second claim, let
\[
u=n_1a_1+\cdots+n_ka_k\in \Sigma_h^+(a_1,\dots,a_k),
\]
so that
\[
n_1,\dots,n_k\in \N
\qquad\text{and}\qquad
n_1+\cdots+n_k=h.
\]
Then
\[
u+M
=
(n_1a_1+\cdots+n_ka_k)+(M_1+\cdots+M_k)
=
\sum_{i=1}^k (n_ia_i+M_i).
\]
Since $n_i\ge 1$, we have
\[
n_i(a_i+M_i)=n_ia_i+n_iM_i=n_ia_i+M_i.
\]
Therefore
\[
u+M
=
\sum_{i=1}^k n_i(a_i+M_i)
\subseteq hA.
\]
This proves the second claim.

Combining the two claims with the inclusion from Proposition~\ref{prop:positiveshelltuple}, we get
\[
rhA
\subseteq
\Sigma_{rh}(a_1,\dots,a_k)+M
\subseteq
X_h+\Sigma_h^+(a_1,\dots,a_k)+M
\subseteq
X_h+hA.
\]
Thus
\[
|X_h|\le \binom{(r+1)(k-1)}{k-1}
\]
for all
\[
h\ge r(k-1)^2+k,
\]
which proves the theorem.
\end{proof}

\begin{proof}[Proof of Theorem~\ref{thm:unboundedlinear}]
For each color $i$, Theorem~\ref{thm:onecolorlinear} implies that $A_i$ is an asymptotic
\[
\left(r,\binom{(r+1)(k_i-1)}{k_i-1}\right)\text{-approximate group}.
\]
Applying the colorwise lifting theorem (Theorem~\ref{thm:colorwise}) with
\[
\ell_i=\binom{(r+1)(k_i-1)}{k_i-1}
\]
for each $i$, we conclude that $\cA$ is a chromatic asymptotic $(r,\ell)$-approximate group with
\[
\ell
=
\prod_{i=1}^q \ell_i
=
\prod_{i=1}^q \binom{(r+1)(k_i-1)}{k_i-1}.
\]
\end{proof}

\begin{corollary}\label{cor:alllinear}
Let $\cA=(A_1,\dots,A_q)$ be a tuple of subsets of an abelian group. Assume that each $A_i$ is a finite union of bounded or unbounded generalized arithmetic progressions. Then for every $r\in \N$, $\cA$ is a chromatic asymptotic $(r,\ell)$-approximate group for some finite $\ell$.

More precisely, after refining each bounded linear set into finitely many singleton unbounded linear sets, if $s_i$ denotes the total number of resulting unbounded linear pieces in $A_i$, then one may take
\[
\ell=
\prod_{i=1}^q \binom{(r+1)(s_i-1)}{s_i-1}.
\]
\end{corollary}

\begin{proof}
Each bounded linear set is a finite union of singleton sets, hence a finite union of unbounded linear sets of the form $P(x;0)$. Therefore each $A_i$ is a finite union of unbounded linear sets. Let $s_i$ denote the number of pieces after this refinement. Then Theorem~\ref{thm:unboundedlinear} applies and yields the stated conclusion.
\end{proof}

\begin{proof}[Proof of Theorem~\ref{thm:inhomogeneous}]
Since $B$ is finite with $|B|=m$, Proposition~\ref{prop:finiteabelian} applied to the set $B$ with $h=1$ yields, for every fixed $r$, a set $Z\subseteq G$ with
\[
|Z|\le \binom{(r+1)(m-1)}{m-1}
\qquad\text{and}\qquad
rB\subseteq Z+B.
\]
Now let $\bh\succeq \bh_0$, where $\bh_0$ is a threshold for the CAAG property of $\cA$. Then there exists $X_{\bh}$ with
\[
|X_{\bh}|\le \ell
\qquad\text{and}\qquad
(r\bh)\cdot \cA\subseteq X_{\bh}+\bh\cdot \cA.
\]
Using abelianness,
\[
r(\bh\cdot \cA+B)=r(\bh\cdot \cA)+rB.
\]
Hence
\[
r(\bh\cdot \cA+B)
\subseteq
X_{\bh}+\bh\cdot \cA+Z+B
=
(X_{\bh}+Z)+(\bh\cdot \cA+B).
\]
Define
\[
Y_{\bh}:=X_{\bh}+Z.
\]
Then
\[
|Y_{\bh}|
\le
|X_{\bh}|\,|Z|
\le
\ell \binom{(r+1)(m-1)}{m-1}.
\]
Thus $S_{\bh}$ is an asymptotic
\[
\left(r,\ell\binom{(r+1)(m-1)}{m-1}\right)\text{-approximate family}.
\]
\end{proof}

\subsection{Representation layers over \texorpdfstring{$\Z$}{Z}}\label{sec:representation}

We now specialize to tuples of finite sets of integers and study the threshold chromatic layers
\[
(\bh\cdot \cA)^{(t)}.
\]
We show that these layers are themselves asymptotically approximate.

We use the following theorem from chromatic sumset theory.

\begin{theorem}[Nathanson \cite{NathansonChromaticSumsets}]\label{thm:nathanson}
Let $\cA=(A_1,\dots,A_q)$ be a normalized tuple of finite sets of integers, and let
\[
a_i^*:=\max(A_i)\ge 1.
\]
Fix $t\in \N$. Then there exist nonnegative integers $c_t,d_t$, finite sets $C_t,D_t\subseteq \Nzero$, and a threshold vector $\bh_t\in \Nzero^q$ such that for every $\bh\succeq \bh_t$,
\[
(\bh\cdot \cA)^{(t)}
=
C_t\cup [c_t,\bh\cdot \ba^*-d_t]\cup (\bh\cdot \ba^*-D_t),
\]
where
\[
\ba^*:=(a_1^*,\dots,a_q^*).
\]
\end{theorem}

\subsubsection{A covering lemma for interval-plus-edge sets}

\begin{lemma}\label{lem:intervalcover}
Let
\[
S_H:=C\cup [c,H-d]\cup (H-D)
\]
for integers $H$ and fixed data $c,d\in \Nzero$, finite sets $C,D\subseteq \Nzero$. Put
\[
L:=H-c-d+1.
\]
Assume $L\ge 1$. Then for every $r\in \N$,
\[
[c,rH-d]\subseteq \bigcup_{j=0}^r (jL+[c,H-d])
\]
provided
\[
H\ge r(c+d)-r.
\]
\end{lemma}

\begin{proof}
Each translate
\[
jL+[c,H-d]
\]
is the interval
\[
[jL+c,\ jL+H-d].
\]
The length of $[c,H-d]$ is
\[
(H-d)-c+1=H-c-d+1=L.
\]
Hence the successive intervals
\[
0L+[c,H-d],\ 1L+[c,H-d],\ \dots,\ rL+[c,H-d]
\]
are consecutive and overlapping or adjacent: the start of the next interval is exactly one more than the end of the previous interval. Therefore their union is the entire interval
\[
[c,\ c+(r+1)L-1].
\]
Thus it suffices to show that
\[
rH-d\le c+(r+1)L-1.
\]
Substituting $L=H-c-d+1$, the right-hand side becomes
\[
c+(r+1)(H-c-d+1)-1.
\]
Subtracting $rH-d$, we get
\[
H-r(c+d)+r.
\]
By hypothesis this is nonnegative. Hence
\[
[c,rH-d]\subseteq [c,c+(r+1)L-1]
=\bigcup_{j=0}^r (jL+[c,H-d]).
\]
This proves the lemma.
\end{proof}

\begin{proof}[Proof of Theorem~\ref{thm:trich}]
Let
\[
S_{\bh}:=(\bh\cdot \cA)^{(t)}.
\]
By Theorem~\ref{thm:nathanson}, there exist fixed data $c_t,d_t,C_t,D_t$ and a threshold vector $\bh_t$ such that for all $\bh\succeq \bh_t$,
\[
S_{\bh}=C_t\cup [c_t,H_{\bh}-d_t]\cup (H_{\bh}-D_t),
\]
where
\[
H_{\bh}:=\bh\cdot \ba^*.
\]
Since every $a_i^*\ge 1$, we have
\[
H_{\bh}=\bh\cdot \ba^*=\sum_{i=1}^q h_i a_i^*\ge \sum_{i=1}^q h_i.
\]
Therefore $H_{\bh}\to \infty$ as $\bh\to \infty$ coordinatewise.

Set
\[
C:=c_t+d_t.
\]
Choose a vector $\bh_{r,t}\succeq \bh_t$ large enough that
\[
H_{\bh}\ge \max\{rC-r,\,C\}
\qquad\text{for all }\bh\succeq \bh_{r,t}.
\]
This is possible because $H_{\bh}\to\infty$ coordinatewise.

Fix $\bh\succeq \bh_{r,t}$.
Since $\bh_{r,t}\succeq \bh_t$ and $r\bh\succeq \bh\succeq \bh_t$, Theorem~\ref{thm:nathanson}
applies both to $S_{\bh}$ and to $S_{r\bh}$.
Write
\[
H:=H_{\bh},
\qquad
I_H:=[c_t,H-d_t],
\qquad
R_H:=H-D_t.
\]
Then
\[
S_{\bh}=C_t\cup I_H\cup R_H.
\]
Define
\[
L:=H-c_t-d_t+1.
\]
Because $H\ge C=c_t+d_t$, we have $L\ge 1$.

Now define
\[
X_{\bh}:=\{0,L,2L,\dots,rL,(r-1)H\}.
\]
Hence
\[
|X_{\bh}|\le r+2.
\]
We claim that
\[
S_{r\bh}\subseteq X_{\bh}+S_{\bh}.
\]
By Theorem~\ref{thm:nathanson} again,
\[
S_{r\bh}=C_t\cup [c_t,rH-d_t]\cup (rH-D_t).
\]
We treat the three pieces separately.

\medskip
\noindent
\textbf{Step 1: the initial finite part.}
Since $C_t\subseteq S_{\bh}$, we have
\[
C_t\subseteq 0+S_{\bh}\subseteq X_{\bh}+S_{\bh}.
\]

\medskip
\noindent
\textbf{Step 2: the central interval.}
By Lemma~\ref{lem:intervalcover} with $c=c_t$, $d=d_t$, and the present value of $H$, we have
\[
[c_t,rH-d_t]\subseteq \bigcup_{j=0}^r (jL+I_H).
\]
But $jL\in X_{\bh}$ for $0\le j\le r$, and $I_H\subseteq S_{\bh}$. Therefore
\[
[c_t,rH-d_t]\subseteq X_{\bh}+S_{\bh}.
\]

\medskip
\noindent
\textbf{Step 3: the terminal finite part.}
Since
\[
R_H=H-D_t\subseteq S_{\bh},
\]
we have
\[
(r-1)H+R_H=(r-1)H+(H-D_t)=rH-D_t\subseteq X_{\bh}+S_{\bh}.
\]
That is,
\[
rH-D_t\subseteq X_{\bh}+S_{\bh}.
\]
This is precisely the terminal finite part of $S_{r\bh}$.

Combining the three steps, we conclude that
\[
S_{r\bh}
=C_t\cup [c_t,rH-d_t]\cup (rH-D_t)
\subseteq X_{\bh}+S_{\bh}.
\]
Thus
\[
((r\bh)\cdot \cA)^{(t)}
\subseteq
X_{\bh}+(\bh\cdot \cA)^{(t)}
\]
with $|X_{\bh}|\le r+2$. This proves the theorem.
\end{proof}

\begin{remark}\label{rem:notdirectfromfinite}
Theorem~\ref{thm:trich} is not a formal consequence of the finite-set covering theorem applied to the sets
\[
S_{\bh}:=(\bh\cdot \cA)^{(t)}.
\]
Indeed, Proposition~\ref{prop:finiteabelian} controls
\[
rS_{\bh}
\]
by translates of $S_{\bh}$ for each fixed finite set $S_{\bh}$, whereas Theorem~\ref{thm:trich} compares the different layers
\[
S_{r\bh}
\qquad\text{and}\qquad
S_{\bh}.
\]
In general one does not have
\[
S_{r\bh}=rS_{\bh}.
\]
Moreover, the finite-set theorem would yield a bound depending on the cardinality of $S_{\bh}$, and hence typically depending on $\bh$, while Theorem~\ref{thm:trich} gives the uniform bound
\[
|X_{\bh}|\le r+2
\]
for all sufficiently large $\bh$.
\end{remark}

\begin{corollary}\label{cor:arbitraryfinite}
Let $\cA'=(A_1',\dots,A_q')$ be any tuple of nonempty finite sets of integers. Fix $t,r\in \N$. Then there exists a threshold vector $\bh_{r,t}$ such that for every $\bh\succeq \bh_{r,t}$, there exists a set $X_{\bh}\subseteq \Z$ with
\[
|X_{\bh}|\le r+2
\]
and
\[
((r\bh)\cdot \cA')^{(t)}
\subseteq
X_{\bh}+(\bh\cdot \cA')^{(t)}.
\]
\end{corollary}

\begin{proof}
Let
\[
a_{i,0}':=\min(A_i')
\qquad (1\le i\le q),
\]
and let
\[
I:=\{i\in [1,q]: |A_i'|\ge 2\},
\qquad
J:=[1,q]\setminus I.
\]

If $I=\varnothing$, then every $A_i'$ is a singleton, say
\[
A_i'=\{a_{i,0}'\}.
\]
In that case every chromatic representation is unique. Hence
\[
(\bh\cdot \cA')^{(t)}=
\begin{cases}
\left\{\sum_{i=1}^q h_i a_{i,0}'\right\}, & t=1,\\[1ex]
\varnothing, & t\ge 2.
\end{cases}
\]
Therefore the desired inclusion holds immediately for every $\bh$: if $t\ge 2$, we may take
\[
X_{\bh}:=\{0\},
\]
while if $t=1$, we take
\[
X_{\bh}:=\left\{(r-1)\sum_{i=1}^q h_i a_{i,0}'\right\}.
\]

Assume now that $I\neq \varnothing$. Define
\[
\tau_{\bh}:=\sum_{i\in J} h_i a_{i,0}'.
\]
Since every singleton color contributes exactly one summand and does not affect multiplicities, we have
\[
(\bh\cdot \cA')^{(t)}
=
\tau_{\bh}+(\bh_I\cdot \cA_I')^{(t)},
\]
where
\[
\cA_I':=(A_i')_{i\in I},
\qquad
\bh_I:=(h_i)_{i\in I}.
\]

Define
\[
d:=\gcd\Bigl(\bigcup_{i\in I}(A_i'-a_{i,0}')\Bigr),
\]
and normalized sets
\[
A_i:=\left\{\frac{a-a_{i,0}'}{d}: a\in A_i'\right\}
\qquad (i\in I).
\]
Then each $A_i$ is finite, normalized, and satisfies
\[
\max(A_i)\ge 1.
\]
Let
\[
\cA:=(A_i)_{i\in I}.
\]
For every $\bh_I$, Nathanson's normalization identity gives
\[
(\bh_I\cdot \cA_I')^{(t)}
=
d\,(\bh_I\cdot \cA)^{(t)}+\sum_{i\in I} h_i a_{i,0}'.
\]
Indeed, for each color $i\in I$, the map
\[
a\mapsto da+a_{i,0}'
\]
is a bijection from $A_i$ onto $A_i'$ preserving the weakly increasing order within color $i$. Applying this map entrywise gives a bijection between the representations counted by $r_{\cA,\bh_I}(n)$ and those counted by $r_{\cA_I',\bh_I}$ at
\[
dn+\sum_{i\in I} h_i a_{i,0}'
\]
respectively.
Hence multiplicities are preserved under this affine normalization.

Apply Theorem~\ref{thm:trich} to the normalized tuple $\cA$. Thus there exists a threshold vector
\[
\bg_{r,t}\in \Nzero^{|I|}
\]
such that for every $\bh_I\succeq \bg_{r,t}$ there exists a set $Y_{\bh_I}\subseteq \Z$ with
\[
|Y_{\bh_I}|\le r+2
\]
and
\[
((r\bh_I)\cdot \cA)^{(t)}
\subseteq
Y_{\bh_I}+(\bh_I\cdot \cA)^{(t)}.
\]

Extend $\bg_{r,t}$ to a vector
\[
\bh_{r,t}\in \Nzero^q
\]
by setting the coordinates in $J$ equal to $0$. Then for every $\bh\succeq \bh_{r,t}$ we have $\bh_I\succeq \bg_{r,t}$. Multiplying the last inclusion by $d$ and translating by
\[
r\sum_{i\in I} h_i a_{i,0}'
\]
yields
\[
((r\bh_I)\cdot \cA_I')^{(t)}
\subseteq
\left(dY_{\bh_I}+(r-1)\sum_{i\in I} h_i a_{i,0}'\right)
+
(\bh_I\cdot \cA_I')^{(t)}.
\]
Finally, translating by
\[
r\tau_{\bh}
\]
gives
\[
((r\bh)\cdot \cA')^{(t)}
\subseteq
X_{\bh}+(\bh\cdot \cA')^{(t)},
\]
where
\[
X_{\bh}:=
dY_{\bh_I}+(r-1)\sum_{i=1}^q h_i a_{i,0}'.
\]
Since
\[
|X_{\bh}|\le |Y_{\bh_I}|\le r+2,
\]
the proof is complete.
\end{proof}

\section{Concluding remarks and further directions}

We studied a chromatic covering framework for tuples of sets in abelian groups and showed that it is stable under homomorphisms and affine normalization, admits exact and finite covering results for translated submonoids and finite-core-plus-submonoid families, and yields chromatic asymptotic approximate-group theorems for semilinear color classes. In the integer setting, Nathanson's interval-plus-edges description of threshold chromatic layers leads to a uniform asymptotic covering theorem for the threshold layers themselves. Thus both chromatic sumsets and threshold chromatic layers fit naturally into a common asymptotic approximate-covering framework. Several natural extensions suggest themselves.

\subsection*{1. Sharper quantitative bounds}
Can we obtain sharper quantitative bounds on the constants?

\subsection*{2. Mixed structural classes}
A natural extension is to formulate a mixed theorem in which different colors are allowed to belong to different structural classes simultaneously. For example, one might allow some colors to be finite, others to be semilinear, and still others to be of the form $F_i+M_i$. 

\subsection*{3. Semilinear and Presburger-definable families}
Over finitely generated abelian groups, semilinear sets are closely related to Presburger-definable sets. The results suggest that chromatic asymptotic approximate-group phenomena may extend beyond finite unions of linear sets to wider classes of eventually semilinear families. A precise structural theorem in this direction may require a more refined analysis of how covering constants behave under Boolean operations and parameterized unions. Even partial results for Presburger-definable color classes would broaden the scope of the theory.

\subsection*{4. Inverse problems}
The present work is primarily of direct type: given a structured tuple, we prove chromatic asymptotic covering. A complementary inverse theory would ask what structural information is forced by the existence of small chromatic covering numbers. For instance, suppose that for some fixed $r$ and $\ell$ one has
\[
(r\bh)\cdot \cA\subseteq X_{\bh}+\bh\cdot \cA
\qquad\text{with}\qquad |X_{\bh}|\le \ell
\]
for all sufficiently large $\bh$. Must the color classes then exhibit semilinear, submonoid-like, or generalized-progression structure? Even partial inverse results of this kind would connect the chromatic theory developed here with the classical inverse theory of small doubling and approximate groups.

\subsection*{5. Dependence on the threshold parameter \texorpdfstring{$t$}{t}}
In the representation-layer results, the threshold $t$ is fixed and the asymptotic behavior is studied as $\bh\to\infty$. It remains to understand how the geometry of the sets
\[
(\bh\cdot \cA)^{(t)}
\]
changes when $t$ is also allowed to vary. A natural question is whether there are useful uniform covering statements in regimes where $t=t(\bh)$ grows slowly with $\bh$, or whether there are phase transitions in the asymptotic shape of threshold layers as the threshold increases.

\subsection*{6. Higher-rank lattice and geometric refinements}
The proof of the finite-color theorem ultimately rests on lattice coverings of simplices by translates of smaller simplices. It may be useful to revisit this geometry more carefully. Better lattice-covering estimates could lead directly to improved explicit constants in both the one-color and chromatic settings. More conceptually, one may seek to identify a geometric model for chromatic covering in which the different colors correspond to independent directions or fibers, thereby making the product structure of the current bounds more transparent.

\subsection*{7. Beyond the abelian setting}
The present theory is formulated in abelian groups, where the chromatic sumset
\[
h_1A_1+\cdots+h_qA_q
\]
is independent of the order of summation and where repeated sumsets behave functorially under homomorphisms and dilations. In nonabelian groups, one would need to decide whether colors come with a prescribed order, whether one studies product sets
\[
A_1^{h_1}\cdots A_q^{h_q},
\]
and which notion of asymptotic covering is most natural. Even very special cases, such as nilpotent groups or ordered product families, may already exhibit nontrivial phenomena.


\end{document}